\documentclass{article}
\usepackage{latexsym, amsmath, amssymb}
\usepackage{graphicx}
\usepackage{epstopdf}
\usepackage{amsmath, amsthm, amscd, amsfonts}
\usepackage{amsmath}
\usepackage{array}
\usepackage{multirow}
\usepackage{amssymb}
\usepackage{float}
\usepackage{enumerate}
\usepackage[a4paper, total={6in, 10in}]{geometry}

\newtheorem{theorem}{Theorem}[section]
\newtheorem{exm}{Example}[section]

\newtheorem{lemma}{Lemma}
\theoremstyle{remark} 

\makeatletter
\newcommand*{\rom}[1]{\expandafter\@slowromancap\romannumeral #1@}
\makeatother

\def\bege{\begin{equation}} \def\ende{\end{equation}}

   \def\begr{\begin{eqnarray}}
\def\endr{\end{eqnarray}} 
 
\def\bege{\begin{equation}} \def\ende{\end{equation}}
\def\begr{\begin{eqnarray}} \def\endr{\end{eqnarray}} \def\bnum{\begin{enumerate}} \def\enum{\end{enumerate}}
\begin{document}

\begin{center}
\textbf{Multiset and Mixed Metric Dimension for Starphene and Zigzag-Edge Coronoid}
\end{center}

\begin{center}
Jia-Bao Liu$^{1}$, Sunny Kumar Sharma$^{2}$, Vijay Kumar Bhat$^{2,\ast}$ and Hassan Raza$^{3}$
\end{center}
$^{1}$School of Mathematics and Physics, Anhui Jianzhu University, Hefei 230601, P. R. China.\\
$^{2}$School of Mathematics, Shri Mata Vaishno Devi University, Katra-$182320$, J\&K, India.\\
$^{3}$Business School, University of Shanghai for Science and Technology, Shanghai 200093, China.\\
$^{1}$liujiabaoad@163.com, $^{2}$sunnysrrm94@gmail.com, $^{a}$vijaykumarbhat2000@yahoo.com, \\
$^{3}$hassan\_raza783@yahoo.com\\\\
\textbf{Abstract} Let $\Gamma=(V,E)$ be a simple connected graph. A vertex $a$ is said to recognize (resolve) two different elements $b_{1}$ and $b_{2}$ from $V(\Gamma)\cup E(\Gamma)$ if $d(a, b_{1})\neq d(a, b_{2}\}$. A subset of distinct ordered vertices $U_{M}\subseteq V(\Gamma)$ is said to be a mixed metric generator for $\Gamma$ if each pair of distinct elements from $V\cup E$ are recognized by some element of $U_{M}$. The mixed metric generator with a minimum number of elements is called a mixed metric basis of $\Gamma$. Then, the cardinality of this mixed metric basis for $\Gamma$ is called the mixed metric dimension of $\Gamma$, denoted by $mdim(\Gamma)$. The concept of studying chemical structures using graph theory terminologies is both appealing and practical. It enables researchers to more precisely and easily examines various chemical topologies and networks. In this paper, we consider two well known chemical structures; starphene $SP_{a,b,c}$ and six-sided hollow coronoid $HC_{a,b,c}$ and respectively compute their multiset dimension and mixed metric dimension. 
$^{\ast}$ Corresponding author.
\section{Introduction and Preliminaries}
Chemical graph theory is the topology branch of mathematical chemistry that employs graph theory to mathematical modeling of chemical phenomena. The complex and large-scale chemical structures possess a challenge to analyze in their natural form. Chemical graph theory is utilized to make these complex structures intelligible. The molecular graph is a graph of a chemical structure in which the vertices represent the atoms, and the edges represent the bonds that connect them.\\

The physical properties of a chemical structure are studied by utilizing a unique mathematical model in which each atom (vertex) possesses a unique identity within the structure. To identify the uniqueness of the entire atoms (vertices) in a chemical structure, we tend to choose those atoms (vertex) set, which would have a unique position to the selected atoms (vertices). This concept in graph theory \cite{ps}, is perceived as a metric basis, and in the applied graph theory \cite{fr}, it is known as a resolving set. \\

Suppose $\Gamma=(V,E)$ is a non-trivial, connected, simple, and finite graph with the edge set $E(\Gamma)$ and the vertex set $V(\Gamma)$. We write $V$ instead of $V(\Gamma)$ and $E$ instead of $E(\Gamma)$ throughout the manuscript when there is no scope for ambiguity. The {\it topological distance} (geodesic) between two vertices $a$ and $b$ in $\Gamma$, denoted by $d(a,b)$, is the length of a shortest $a-b$ path between the vertices $a$ and $b$ in the graph $\Gamma$. The number of edges that are incident to a vertex $\alpha$ of a graph $\Gamma$ is known as its $degree$ (or valency) and is denoted by $d_{\alpha}$. The minimum degree and the maximum degree of $\Gamma$ are denoted by $\delta(\Gamma)$ and $\Delta(\Gamma)$, respectively. An $independent$ $set$ is a set of vertices in $\Gamma$, in which no two vertices are adjacent.\\

The distance between a vertex $a\in V(\Gamma)$, and an edge $e = bc\in E(\Gamma)$ is defined as $d(a,e)=d(a,bc)=min\{d(a,b),d(a,c)\}$. A vertex $z\in V(\Gamma)$ resolves a pair of elements say $f, f^\prime \in V(\Gamma)\cup E(\Gamma)$ if $d(z,f)\neq d(z,f^{\prime})$. A set $U_{M}\subseteq V(\Gamma)$ is a mixed metric generator (MMG) when $f, f^\prime \in V(\Gamma)\cup E(\Gamma)$ is resolved by some vertex of $U_{M}$. The number of elements in the set $U_{M}$ is called mixed resolving basis, and it's minimum cardinality is termed as {\it mixed metric dimension}, denoted as $mdim(\Gamma)$. Equivalently, for an ordered subset of vertices $U_{M}=\{z_{1}, z_{2}, z_{3},...,z_{k}\}$, the $k^{th}$-mixed code (coordinate) of an element $z$ in $V(\Gamma)\cup E(\Gamma)$ is;
\begin{center}
  \begin{eqnarray*}
  \gamma_{M}(z|U_{M})&=&(d(z_{1},z),d(z_{2},z),d(z_{3},z),...,d(z_{k},z))
  \end{eqnarray*}
\end{center}
Then we say that, the set $U_{M}$ is a MMG for $\Gamma$, if $\gamma_{M}(p_{1}|U_{M})\neq \gamma_{M}(p_{2}|U_{M})$, for any pair of elements $p_{1}, p_{2} \in V\cup E$ with $p_{1} \neq p_{2}$. A set of distinct vertices $U^{i}_{M}$ (ordered) in $\Gamma$ is said to be an {\it independent mixed metric generator} (IMMG) for $\Gamma$ if $U^{i}_{M}$ is (i) independent set and (ii) MMG.\\

The concept of mixed metric dimension is actually the combination of the two concepts, those are metric dimension (denoted by $dim(\Gamma)$) which resolves vertices of the given graphs, and edge metric dimension (denoted by $edim(\Gamma)$) which resolves edges of the given graphs $\Gamma$. It was introduced by Kelenc et al. \cite{ssv1}, and recently it is studied for different families of graphs and in general as well. For further in depth of this concept, we refer the readers to \cite{H2019,H2020,HR2020,JS2021,JS12021,ssv1}. \\

Multisets (MSs for short and these are the sets with repeating elements) are of interest in computer science, physics, logic, philosophy, linguistic and mathematics. The multiplicity of units is the most important property of MS, allowing us to distinguish it from a set and regard it as a qualitatively new mathematical phenomenon. The MS $M=\{a,a,a,a,b,b,b,c,c,d\}$ is said to contain the element $a$ four times, the element $b$ three times, the element $c$ twice, and the element $d$ once. Since, the ordering of elements in MSs are ignored, the above MS $M$ is the same as $\{a,a,b,a,a,b,b,d,c,c\}$ and $\{c,c,d,a,b,b,a,a,a,b\}$. So, the MS is generally represented by identifying the number of times various types of elements appear in it. Thus, the MS $K$ can be written as $K=\{4.\bar{x},3.\bar{y},2.\bar{z},1.\bar{a}\}$. The numbers 1,2, 3, and 4 are said to be the $repetition$ $numbers$ of the MS $K$. A set, in particular, is an MS with all repetition numbers equal to one.\\

For a vertex $z$ and a set $U_{ms}=\{z_{1}, z_{2}, z_{3},...,z_{s}\}$ of $\Gamma$, we refer to the $s$-multiset $mr_{c}(z|U_{ms})=\{d(z_{1},z), d(z_{2},z), d(z_{3},z),...,d(z_{s},z)\}$ as the $multirepresentation$ (MR) of $z$ with respect to $U_{ms}$. A subset of vertices $U_{ms}$ in $V(\Gamma)$ is said to be a $multiresolving$ $set$ (MRS) for $\Gamma$ if for every $y,z \in V(\Gamma)$ with $y \neq z$, we have $mr_{c}(y|U_{ms})\neq mr_{c}(z|U_{ms})$. An MRS having minimum number of vertices is called a $multibasis$ for $\Gamma$ and the cardinality of the multibasis set is the $multiset$ $dimension$ (MSD) $msdim(\Gamma)$ of $\Gamma$ \cite{msd}.\\

Clearly, MRS does not exist for all connected graphs and so $msdim(\Gamma)$ is not defined for all non-trivial connected graphs $\Gamma$. For example, the complete graph $K_{n}$; $n\geq3$ has no MRS and MSD for $K_{n}$ is not defined i.e., $msdim(K_{n})=\infty$. On the other hand, if $\Gamma$ is a simple connected graph with $|V(\Gamma)|=n$, for which $msdim(\Gamma)$ is defined, then every MRS of $\Gamma$ is an RS of $\Gamma$, and thus $1\leq dim(\Gamma)\leq msdim(\Gamma) \leq n$.\\

The concept of metric dimension was put forward in \cite{ps}, and later this concept captured the attention of numerous researchers \cite{cee,1,svs,sv,ssv}. Due to the applicability perspective of this concept including geographical routing protocols \cite{pil}, robot navigation \cite{krr}, telecommunications \cite{zf}, combinatorial optimization \cite{co}, network discovery and verification \cite{zf}, connected joints in-network and chemistry \cite{cee}, it has been widely acknowledged. Kelenc et al. \cite{emd} introduced and initiated the research of a new variation of metric dimension in non-trivial connected graphs termed as the edge metric dimension, which focuses on uniquely identifying the graph\textquotesingle s edges. Some of the results concerning the edge metric dimension are studied in \cite{1ft,I2020,ssv,YZ2020}. The computational complexity and NP-hardness of all the resolvability parameter considered in this article are addressed in \cite{1,2}.
\subsection{Chemical Structures}

Benzene is an organic molecule possessing the chemical formula $C_{6}H_{6}$. It is utilized as a solvent in many commercial, scientific, and industrial activities. Benzene is an essential component of gasoline that may also be detected in crude oil. It is used to make plastics, dyes, pharmaceuticals products, detergents, rubber lubricants, insecticides, resins, and synthetic textiles. When benzene rings are joined together to produce more prominent polycyclic aromatic compounds, they are called polycyclic aromatic compounds.\\

Starphenes are two-dimensional polyaromatic hydrocarbons with three acene \cite{acn} arms joined by a single benzene ring. These structures can be employed as NOR \cite{nor} or NAND \cite{nand}  gates in single-molecule electronics. Moreover, as a type of 2D  polyaromatic hydrocarbons, starphenes may be appealing as a component in organic light-emitting diodes (OLEDs) or organic field-effect transistors (OFETs) \cite{oled}.\\

Brunvoll et al. \cite{edd} proposed the term coronoid to describe its possible relationship with benzenoid. A benzenoid with a hole in the middle is called a coronoid. Coronoid is an organic chemistry-based polyhex system. The zigzag-edge coronoids, as shown in Fig. 1, denoted by $HC_{a,b,c}$, can be considered as a structure formed by fusing six linear polyacenes segments into a closed loop. A hollow coronoid \cite{alikom2} is another name for this structure. It is classified as catacondensed coronoids, and it belongs to the primitive coronoids sub-cluster \cite{14}. A hollow coronoid as depicted in Fig. 1, has six sides viz., $a$, $b$, $c$, $a$, $b$, $c$.\\

Many studies have considered hollow coronoid. For example, \cite{16} discusses the relationship between the polyhex system and the hollow coronoid; \cite{111} discusses polynomial research; \cite{12} discusses the mathematical study of the coronoid and similar structures, and \cite{14,13} discusses the hollow coronoids and their generalizations.\\

The notion of metric dimension and its variants is considered for numerous chemical structures due to its applicability in chemical sciences. The vertex resolvability of $H$-Napthalenic nanotubes and $VC_{5}C_{7}$ were discussed in \cite{17n}. In \cite{45n}, the partition dimension and the metric dimension of 2D lattices of the infinite version of $HAC_{5}C_{6}C_{7}$, $HAC_{5}C_{7}$, and $HC_{5}C_{7}$ nanotubes are obtained. The upper bounds for the minimum vertex resolving sets of silicate star networks are established in \cite{41n}. The metric dimension of the 2D lattice of $\alpha$-Boron nanotubes is discussed in \cite{16n}, and the upper bounds for the minimum vertex metric generators of cellulose network are obtained in \cite{40n}. A set of distinct ordered vertices $U^{i}$ in $\Gamma$ is said to be an Independent vertex resolving set(IVRS) for $\Gamma$ if $U$ is both independent as well as resolving set. A set of distinct vertices $U^{i}_{E}$ (ordered) in $\Gamma$ is said to be an Independent edge resolving set (IERS) for $\Gamma$ if $U^{i}_{E}$ is both independent as well as edge metric generator (EMG). \\

In terms comibinatorial properties of the mixed metric generators, it is quite clear that any mixed metric generator is edge as well metric generator. So the following relation is presented as\\

 \noindent{Remark 1}\cite{mmd}:  For any graph $\Gamma$ of order $n$
 \begin{equation*}
mdim(\Gamma)\geq max\{dim(\Gamma), edim(\Gamma)\}
 \end{equation*}

It is imperative to note that a vertex by itself cannot form a mixed metric generator. It means that mixed metric generator must be $mdim(\Gamma)\geq 2$. More cleary it can be written as. \\

\noindent{Remark 2}\cite{mmd}:  For any graph $\Gamma$ of order $n$
\begin{equation*}
2	\leq mdim(\Gamma)\leq n.
\end{equation*}
In the introductory article by Kelenc et al. \cite{mmd} regarding the mixed metric dimension, the authors posed several questions one of which was which families of graphs have $mdim(\Gamma)=dim (\Gamma)$,  $mdim(\Gamma)=edim(\Gamma)$.  In this paper we not only compute mixed metric dimension of studied graph but also compare the results with the metric and edge metric dimension. \\

Koam et al. \cite{alikom2}, recently studied the vertex and edge resolvability for hollow coronoid $HC_{a,b,c}$. They proved that $HC_{a,b,c}$ is the family of plane graphs with the constant vertex and edge metric dimension(EMD). The metric dimension and the EMD for $HC_{a,b,c}$ is as follows:
\begin{theorem}
	$dim(HC_{a,b,c})=3$, for $a,b,c\geq 3$.
\end{theorem}
\begin{theorem}
	$edim(HC_{a,b,c})=3$, for $a,b,c\geq 3$.
\end{theorem}

In \cite{msd}, Saenpholphat introduced the concept of MRS in graphs and he proved that there exists no connected graph $\Gamma$ for which $msdim(\Gamma)=2$. Also, the MSD of bipartite graphs, paths $P_{n}$, complete graphs $K_{n}$, and cycles $C_{n}$ were computed. Now, we present some basic results and observations regarding MSD in order to carry out our main work.

\begin{lemma}\label{L1}\cite{msd}
$msdim(P_{n})=1$ iff $P_{n}$ is a path graph of length $n$.
\end{lemma}

\begin{lemma}\label{L2}\cite{msd}
Let $\Gamma \neq P_{n}$ be a graph. Then $msdim(\Gamma)\geq3$.
\end{lemma}

The present paper is organized as follows. In Sect. 2 concepts and theory related to mixed metric dimension, edge metric dimension, vertex metric dimension, and independence in their respective resolving sets have been discussed. In Sect. 3 we study the mixed metric dimension and independence in the mixed resolving set of $HC_{a,b,c}$. In Sect. 4, the exact multiset dimensions for $SP_{a,b,c}$ has been obtained. Finally, the future work and conclusion of this work is presented in section 5.

\section{Minimum Mixed Resolving Number of $HC_{a,b,c}$}
In this section, we obtain the mixed metric dimension and IMMG for $HC_{a,b,c}$.\\\\
\textbf{Hollow Coronoid $HC_{a,b,c}$:} Hollow coronoid $HC_{a,b,c}$ comprises of six sides, in which three sides $(a,b,c)$ are symmetric to other three sides $(a,b,c)$ as shown in Fig. 1. This means that $HC_{a,b,c}$ has three linear polyacenes segments consist of $a$, $b$, and $c$ number of benzene rings. It consists of $2a+2b+2c-6$ number of faces having six sides, a face having $4a+4b+4c-18$ sides, and a face having $4a+4b+4c-6$ sides. $HC_{a,b,c}$ has $4a+4b+4c-12$ number of vertices of degree two and $4a+4b+4c-12$ number of vertices of degree three. From this, we find that $\delta(HC_{a,b,c})=2$ and $\Delta(HC_{a,b,c})=3$. The vertex set and the edge set of $HC_{a,b,c}$, are denoted by $V(HC_{a,b,c})$ and $E(HC_{a,b,c})$ respectively. Moreover, the cardinality of edges and vertices in $HC_{a,b,c}$ is given by $|E(HC_{a,b,c})|=10(a+b+c-3)$ and $|V(HC_{a,b,c})|=8(a+b+c-3)$, respectively. The edge and vertex set of $HC_{a,b,c}$ are describe as follows: $V(HC_{a,b,c})=\{p_{1,g}, p_{2,g}|1\leq g\leq 2a-1\}\cup\{q_{1,g}, q_{2,g}|1\leq g\leq 2c-1\}\cup\{r_{1,g}, r_{2,g}|1\leq g\leq 2b-1\}\cup \{s_{1,g}, s_{2,g}|1\leq g\leq 2a-3\}\cup \{u_{1,g}, t_{2,g}|1\leq g\leq 2b-3\}\cup\{t_{1,g}, u_{2,g}|1\leq g\leq 2c-3\}$ \\\\
and\\\\
$E(HC_{a,b,c})=\{p_{1,g}p_{1,g+1}, p_{2,g}p_{2,g+1}|1\leq g\leq 2a-2\}\cup\{q_{1,g}q_{1,g+1}, q_{2,g}q_{2,g+1}|1\leq g\leq 2c-2\}\cup\{r_{1,g}r_{1,g+1}, r_{2,g}r_{2,g+1}|1\leq g\leq 2b-2\}\cup \{s_{1,g}s_{1,g+1}, s_{2,g}s_{2,g+1}|1\leq g\leq 2a-4\}\cup \{u_{1,g}u_{1,g+1}, t_{2,g}t_{2,g+1}|1\leq g\leq 2b-4\}\cup\{t_{1,g}t_{1,g+1}, u_{2,g}u_{2,g+1}|1\leq g\leq 2c-4\}\cup\{p_{1,2g}s_{1,2g-1}, p_{2,2g}s_{2,2g-1}|1 \leq g \leq a-1\}\cup\{q_{1,2g}t_{1,2g-1}, q_{2,2g}u_{2,2g-1}|1 \leq g \leq c-1\}\cup\{r_{1,2g}u_{1,2g-1}, r_{2,2g}t_{2,2g-1}|1 \leq g \leq b-1\}\cup \{p_{1,1}q_{2,1}, s_{1,1}u_{2,1}, \\ p_{1,i}q_{1,1}, s_{1,i-2}t_{1,1}, q_{1,j}r_{1,1}, t_{1,j-2}u_{1,1}, r_{1,k}p_{2,i}, u_{1,k-2}s_{2,l}, p_{2,1}r_{2,k}, s_{2,1}t_{2,k-2}, r_{2,1}q_{2,j}, t_{2,1}u_{2,j-2}\}$.
\begin{center}
  \begin{figure}[h!]
  \centering
   \includegraphics[width=3.0in]{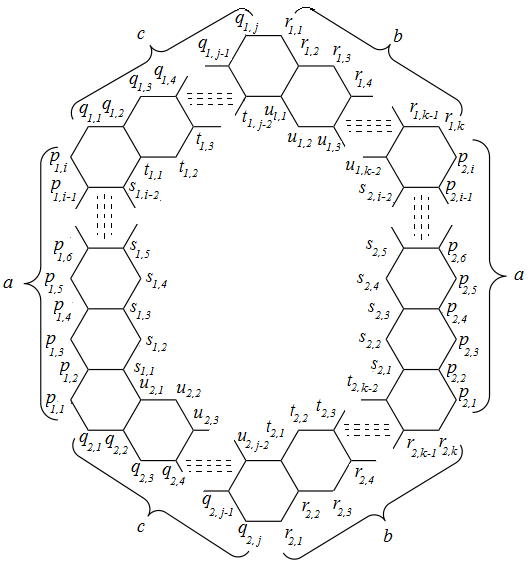}
  \caption{$HC_{a,b,c}$}\label{p2}
\end{figure}
\end{center}
We name the vertices on the cycle $p_{1,1},...,p_{1,i},q_{1,1},...,q_{1,j}r_{1,1},...,r_{1,k},p_{2,1},...,p_{2,i},q_{2,1},...,q_{2,j}r_{2,1},...,\\r_{2,k}$ as the outer $pqr$-cycle vertices and the vertices on the cycle $s_{1,1},...,s_{1,i-2},t_{1,1},...,t_{1,j-2}u_{1,1},...,\\u_{1,k-2},s_{2,1},...,s_{2,i-2},t_{2,1},...,t_{2,k-2}u_{2,1},...,u_{2,j-2}$ as the inner $stu$-cycle vertices. In vertices, $p_{1,i}$, $p_{2,i}$, $q_{1,j}$, $q_{2,j}$, $r_{1,k}$, and $r_{2,k}$, the indices $i=2a-1$, $j=2c-1$ and $k=2b-1$. In the next result, we determine the MMD of $HC_{a,b,c}$.

\begin{theorem}
For $a,b,c\geq 4$, we have $mdim(HC_{a,b,c})=3$.
\end{theorem}

\begin{proof}
In order to show that $mdim(HC_{a,b,c}) \leq 3$, we have to construct a mixed metric generator for $HC_{a,b,c}$ having cardinality less than or equal to three. Suppose $U_{M}=\{p_{1,1}, r_{1,1}, p_{2, 1}\}$ be a set of distinct vertices from $HC_{a,b,c}$. We claim that $U_{M}$ is a mixed metric generator for $HC_{a,b,c}$. Now, to obtain $mdim(HC_{a,b,c}) \leq 3$, we can give mixed codes to every element of $V(HC_{a,b,c})\cup E(HC_{a,b,c})$ with respect to the set $U_{M}$. Firstly, we give mixed codes to every vertex of $HC_{a,b,c}$.\\\\
For the vertices $\{p_{1,g}|1 \leq g \leq 2a-1\}$, the mixed codes are as follows:
\begin{eqnarray*}
   \gamma_{M}(p_{1,g}|U_{M}) = \begin{cases}
                         (g-1,2a+2c-g-1,2b+2c-1),   & g=1;\\
                         (g-1,2a+2c-g-1,2b+2c+g-4), & 2 \leq g \leq 2a-1.
                                             \end{cases}
\end{eqnarray*}
For the vertices $\{q_{1,g}|1 \leq g \leq 2c-1\}$, the mixed codes are as follows:
\begin{eqnarray*}
   \gamma_{M}(q_{1,g}|U_{M}) = \begin{cases}
                         (2a+g-2, 2c-g, 2a+2b+2c-g-5), & 1 \leq g \leq 2c-2;\\
                         (2a+g-2, 2c-g, 2a+2b-2),      & g=2c-1.
                                             \end{cases}
\end{eqnarray*}
For the vertices $\{r_{1,g}|1 \leq g \leq 2b-1\}$, the mixed codes are as follows:
\begin{eqnarray*}
   \gamma_{M}(r_{1,g}|U_{M}) = \begin{cases}
                         (2a+2c-2,g-1,2a+2b-g-2),    & g=1;\\
                         (2a+2c+g-5,g-1,2a+2b-g-2),  & 2 \leq g \leq 2b-1.
                                             \end{cases}
\end{eqnarray*}
For the vertices $\{p_{2,g}|1 \leq g \leq 2a-1\}$, the mixed codes are as follows:
\begin{eqnarray*}
   \gamma_{M}(p_{2,g}|U_{M}) = \begin{cases}
                         (2b+2c-1,2a+2b-g-2,g-1),    & g=1;\\
                         (2b+2c+g-4,2a+2b-g-2,g-1),  & 2 \leq g \leq 2a-1.
                                             \end{cases}
\end{eqnarray*}
For the vertices $\{q_{2,g}|1 \leq g \leq 2c-1\}$, the mixed codes are as follows:
\begin{eqnarray*}
   \gamma_{M}(q_{2, g}|U_{M})= \begin{cases}
                         (g,2b+2c-g-1,2a+2c-1),    & g=1;\\
                         (g,2b+2c-g-1,2a+2c+g-4),  & 2 \leq g \leq 2c-1.
                                             \end{cases}
\end{eqnarray*}
For the vertices $\{r_{2,g}|1 \leq g \leq 2b-1\}$, the mixed codes are as follows:
\begin{eqnarray*}
   \gamma_{M}(r_{2,g}|U_{M})= \begin{cases}
                         (2c+g-1,2b-g,2a+4b-g-5), & 1\leq g \leq 2b-2;\\
                         (2c+g-1,2b-g,2a+2b-2),   & g=2b-1.
                                             \end{cases}
\end{eqnarray*}
For the vertices $\{s_{1,g}|1 \leq g \leq 2a-3\}$, the mixed codes are as follows:
\begin{eqnarray*}
   \gamma_{M}(s_{1,g}|U_{M})= \begin{cases}
                         (g+1,2a+2c-g-3,2b+2c+g-4),  & 1\leq g\leq 2a-3.
                                             \end{cases}
\end{eqnarray*}
For the vertices $\{t_{1,g}|1 \leq g \leq 2c-3\}$, the mixed codes are as follows:
\begin{eqnarray*}
   \gamma_{M}(t_{1,g}|U_{M})=\begin{cases}
                         (2a+g-2,2c-g,2a+2b+2c-g-7),  & 1\leq g\leq 2c-3.
                                             \end{cases}
\end{eqnarray*}
For the vertices $\{u_{1,g}|1 \leq g \leq 2b-1\}$, the mixed codes are as follows:
\begin{eqnarray*}
   \gamma_{M}(u_{1,g}|U_{M})=\begin{cases}
                         (2a+2c+g-5,g+1,2a+2b-g-4),  & 1\leq g\leq 2b-3.
                                             \end{cases}
\end{eqnarray*}
For the vertices $\{s_{2,g}|1 \leq g \leq 2a-3\}$, the mixed codes are as follows:
\begin{eqnarray*}
   \gamma_{M}(s_{2,g}|U_{M})=\begin{cases}
                         (2a+2c+g-6,2a+2b-g-4,g+1),  & 1\leq g\leq 2a-3.
                                             \end{cases}
\end{eqnarray*}
For the vertices $\{u_{2,g}|1\leq g\leq 2c-3\}$, the mixed codes are as follows:
\begin{eqnarray*}
   \gamma_{M}(u_{2,g}|U_{M})=\begin{cases}
                         (g+2,2a+2c+g-4,2b+2c-g-3),  & 1\leq g \leq 2c-3.
                                             \end{cases}
\end{eqnarray*}
For the vertices $\{t_{2,g}|1\leq g\leq 2b-3\}$, the mixed codes are as follows:
\begin{eqnarray*}
   \gamma_{M}(t_{2,g}|U_{M})=\begin{cases}
                         (2c+g-1,2a+4b-g-7,2b-g),  & 1\leq g \leq 2b-3.
                                             \end{cases}
\end{eqnarray*}
Next, we give mixed codes to every edge of $HC_{a,b,c}$. For the edges $\{p_{1,g}p_{1,g+1}|1 \leq g \leq 2a-2\}$, the mixed codes are as follows:
\begin{eqnarray*}
   \gamma_{M}(p_{1,g}p_{1,g+1}|U_{M})= \begin{cases}
                         (g-1,2a+2c-g-2,2b+2c-2),    & g=1;\\
                         (g-1,2a+2c-g-2,2b+2c+g-4),  & 2 \leq g \leq 2a-2.
                                             \end{cases}
\end{eqnarray*}
For the edges $\{q_{1,g}q_{1,g+1}|1 \leq g \leq 2c-2\}$, the mixed codes are as follows:
\begin{eqnarray*}
   \gamma_{M}(q_{1,g}q_{1,g+1}|U_{M})= \begin{cases}
                         (2a+g-2,2c-g-1,2a+2b+2c-g-6),  & 1 \leq g \leq 2c-3;\\
                         (2a+g-2,2c-g-1,2a+2b-3),       & 2 \leq g \leq 2c-2.
                                             \end{cases}
\end{eqnarray*}
For the edges $\{r_{1,g}r_{1,g+1}|1 \leq g \leq 2b-2\}$, the mixed codes are as follows:
\begin{eqnarray*}
   \gamma_{M}(r_{1,g}r_{1,g+1}|U_{M})= \begin{cases}
                         (2a+2c-3,g-1,2a+2b-g-3),   & g=1;\\
                         (2a+2c+g-5,g-1,2a+2b-g-3), & 2 \leq g \leq 2b-2.
                                             \end{cases}
\end{eqnarray*}
For the edges $ \{s_{1,g}s_{1,g+1}|1 \leq g \leq 2a-4\}$, the mixed codes are as follows:
\begin{eqnarray*}
   \gamma_{M}(s_{1,g}s_{1,g+1}|U_{M})= \begin{cases}
                         (g+1,2a+2c-g-4,2b+2c+g-4), & 1 \leq g \leq 2a-4.
                                             \end{cases}
\end{eqnarray*}
For the edges $\{t_{1,g}t_{1,g+1}|1 \leq g \leq 2c-4\}$, the mixed codes are as follows:
\begin{eqnarray*}
   \gamma_{M}(t_{1,g}t_{1,g+1}|U_{M})= \begin{cases}
                         (2a+g-2,2c-g-1,2a+2b+2c-g-8), & 1 \leq g \leq 2c-4.
                                             \end{cases}
\end{eqnarray*}
For the edges $\{u_{1,g}u_{1,g+1}|1 \leq g \leq 2b-4\}$, the mixed codes are as follows:
\begin{eqnarray*}
   \gamma_{M}(u_{1,g}u_{1,g+1}|U_{M})= \begin{cases}
                         (2a+2c+g-5,g+1,2a+2b-g-5), & 1 \leq g \leq 2b-4.
                                             \end{cases}
\end{eqnarray*}
For the edges $\{p_{2,g}p_{2,g+1}|1 \leq g \leq 2a-2\}$, the mixed codes are as follows:
\begin{eqnarray*}
   \gamma_{M}(p_{2,g}p_{2,g+1}|U_{M})= \begin{cases}
                         (2b+2c-2,2a+2b-g-3,g-1),   & g=1;\\
                         (2b+2c+g-4,2a+2b-g-3,g-1), & 2 \leq g \leq 2a-2.

                                             \end{cases}
\end{eqnarray*}
For the edges $\{q_{2,g}q_{2,g+1}|1 \leq g \leq 2c-2\}$, the mixed codes are as follows:
\begin{eqnarray*}
   \gamma_{M}(q_{2,g}q_{2,g+1}|U_{M})= \begin{cases}
                         (g,2a+2c-2,2b+2c-g-2),     & g=1;\\
                         (g,2a+2c+g-4,2b+2c-g-2), & 2 \leq g \leq 2c-2.

                                             \end{cases}
\end{eqnarray*}
For the edges $\{r_{2,g}r_{2,g+1}|1 \leq g \leq 2b-2\}$, the mixed codes are as follows:
\begin{eqnarray*}
   \gamma_{M}(r_{2,g}r_{2,g+1}|U_{M})= \begin{cases}
                         (2c+g-1,2a+4b-g-6,2b-g-1),  & 1\leq g \leq 2b-3;\\
                         (2c+g-1,2a+2b-3,2b-g-1),    & g=2b-2.
                                             \end{cases}
\end{eqnarray*}
For the edges $\{s_{2,g}s_{2,g+1}|1 \leq g \leq 2a-4\}$, the mixed codes are as follows:
\begin{eqnarray*}
   \gamma_{M}(s_{2,g}s_{2,g+1}|U_{M})= \begin{cases}
                         (2b+2c+g-4,2a+2b-g-5,g+1),    & 1 \leq g \leq 2a-4.
                                             \end{cases}
\end{eqnarray*}
For the edges $\{t_{2,g}t_{2,g+1}|1 \leq g \leq 2b-4\}$, the mixed codes are as follows:
\begin{eqnarray*}
   \gamma_{M}(t_{2,g}t_{2,g+1}|U_{M})= \begin{cases}
                         (2c+g-1,2a+4b-g-8,2b-g-1),    & 1 \leq g \leq 2b-4.
                                             \end{cases}
\end{eqnarray*}
For the edges $\{u_{2,g}u_{2,g+1}|1 \leq g \leq 2c-4\}$, the mixed codes are as follows:
\begin{eqnarray*}
   \gamma_{M}(u_{2,g}u_{2,g+1}|U_{M})= \begin{cases}
                         (g+2,2a+2c+g-4,2b+2c-g-4),    & 1 \leq g \leq 2c-4.
                                             \end{cases}
\end{eqnarray*}
For the edges $\{\eta_{1}=p_{1,1}q_{2,1}, \eta_{2}=s_{1,1}u_{2,1}, \eta_{3}=p_{1,i}q_{1,1}, \eta_{4}=s_{1,i-2}t_{1,1}, \eta_{5}=q_{1,j}r_{1,1}, \eta_{6}=t_{1,j-2}u_{1,1}, \\ \eta_{7}=r_{1,k}p_{2,i}, \eta_{8}=u_{1,k-2}s_{2,l}, \eta_{9}=p_{2,1}r_{2,k}, \eta_{10}=s_{2,1}t_{2,k-2}, \eta_{11}=r_{2,1}q_{2,j}, \eta_{12}=t_{2,1}u_{2,j-2}\}$, the mixed codes are shown in Table \ref{T1}.
\begin{table}[h]
\begin{center}
 \begin{tabular}{|c|c|c|c|}
 \hline
 Edges & Codes & Edges & Codes\\
 \hline
 $\eta_{1}$  & (0,2a+2c-2,2b+2c-4)   & $\eta_{7}$  & (2a+2b+2c-6,2b-2,2a-2)\\
 \hline
 $\eta_{2}$  & (2,2a+2c-4,2b+2c-6)   & $\eta_{8}$  & (2a+2b+2c-8,2b-2,2a-2)\\
 \hline
 $\eta_{3}$  & (2a-2,2c-1,2a+2b+2c-6)& $\eta_{9}$  & (2b+2c-2,2a+2b-3,0)\\
 \hline
 $\eta_{4}$  & (2a-2,2c-1,2a+2b+2c-8)& $\eta_{10}$ & (2b+2c-4,2a+2b-5,2)\\
 \hline
 $\eta_{5}$  & (2a+2c-3,0,2a+2b-3)   & $\eta_{11}$ & (2c-1,2a+4b-6,2b-1)\\
 \hline
 $\eta_{6}$  & (2a+2c-5,2,2a+2b-5)   & $\eta_{12}$ & (2c-1,2a+4b-8,2b-1)\\
 \hline
 \end{tabular}
 \caption{\label{T1}Mixed codes for the edges $\eta_{i}$; $1\leq i\leq 12$.}
 \end{center}
 \end{table}

\hspace{-1.3em}For the edges $\{p_{1,2g}s_{1,2g-1}|1 \leq g \leq a-1\}$, the mixed codes are as follows:
\begin{eqnarray*}
   \gamma_{M}(p_{1,2g}s_{1,2g-1}|U_{M})= \begin{cases}
                         (2g-1,4a+2c-2g-12,4b+2c+2g-13),    & 1 \leq g \leq a-1.
                                             \end{cases}
\end{eqnarray*}
For the edges $\{q_{1,2g}t_{1,2g-1}|1 \leq g \leq c-1\}$, the mixed codes are as follows:
\begin{eqnarray*}
   \gamma_{M}(q_{1,2g}t_{1,2g-1}|U_{M})= \begin{cases}
                         (4a+2g-13,4c-2g-8,4a+2b+2c-2g-16),    & 1 \leq g \leq c-1.
                                             \end{cases}
\end{eqnarray*}
For the edges $\{r_{1,2g}u_{1,2g-1}|1 \leq g \leq b-1\}$, the mixed codes are as follows:
\begin{eqnarray*}
   \gamma_{M}(r_{1,2g}u_{1,2g-1}|U_{M})= \begin{cases}
                         (4a+2c+2g-16,2g-1,4a+2b-2g-13),    & 1 \leq g \leq b-1.
                                             \end{cases}
\end{eqnarray*}
For the edges $\{p_{2,2g}s_{2,2g-1}|1 \leq g \leq a-1\}$, the mixed codes are as follows:
\begin{eqnarray*}
   \gamma_{M}(p_{2,2g}s_{2,2g-1}|U_{M})= \begin{cases}
                         (4c+2b+2g-13,4a+2b-2g-13,2g-1),    & 1 \leq g \leq a-1.
                                             \end{cases}
\end{eqnarray*}
For the edges $\{r_{2,2g}t_{2,2g-1}|1 \leq g \leq b-1\}$, the mixed codes are as follows:
\begin{eqnarray*}
   \gamma_{M}(r_{2,2g}t_{2,2g-1}|U_{M})= \begin{cases}
                         (4c+2g-10,4a+4b-2g-16,4b-2g-8),    & 1 \leq g \leq b-1.
                                             \end{cases}
\end{eqnarray*}
For the edges $\{q_{2,2g}u_{2,2g-1}|1 \leq g \leq c-1\}$, the mixed codes are as follows:
\begin{eqnarray*}
   \gamma_{M}(q_{2,2g}u_{2,2g-1}|U_{M})= \begin{cases}
                         (2g,4a+2c+2g-15,4b+2c-2g-10),    & 1 \leq g \leq c-1.
                                             \end{cases}
\end{eqnarray*}
Now, from these sets of mixed metric codes for the graph $HC_{a,b,c}$, we find that $|V_{1}|=|V_{4}|=2a-1$, $|V_{2}|=|V_{5}|=2c-1$, $|V_{3}|=|V_{6}|=2b-1$, $|W_{1}|=|W_{4}|=2a-3$, $|W_{2}|=|W_{5}|=2c-3$, $|W_{3}|=|W_{6}|=2b-3$, $|P_{1}|=|P_{2}|=2a-2$, $|Q_{1}|=|Q_{2}|=2c-2$, $|R_{1}|=|R_{2}|=2b-2$, $|S_{1}|=|S_{2}|=2a-4$, $|T_{1}|=|U_{2}|=2c-4$, $|T_{2}|=|U_{1}|=2b-4$, $|PS_{1}|=|PS_{2}|=a-1$, $|RU_{1}|=|RT_{2}|=b-1$, $|QT_{1}|=|QU_{2}|=c-1$, and $|V_{7}|=12$. We see that the sum of all of these cardinalities is equal to $|E(HC_{a,b,c})|+ |V(HC_{a,b,c})|$ and which is $18(a+b+c-3)$. Moreover, all of these sets are pairwise disjoint, which implies that $mdim(HC_{a,b,c})\leq 3$. From this, we conclude that the set $U_{M}$ is a mixed metric generator for $HC_{a,b,c}$ of cardinality three. Now, by using equation (1), we find that $mdim(HC_{a,b,c})\geq3$. Hence, $mdim(HC_{a,b,c})=3$, which concludes the theorem.
\end{proof}

\hspace{-5.0mm}In terms of minimum IMMG, we have the following result

\begin{theorem}
For $a,b,c\geq4$, the graph $HC_{a,b,c}$ has an IMMG with cardinality three.
\end{theorem}

\begin{proof}
To show that, for hollow coronoid $HC_{a,b,c}$, there exists an IMMG $U_{M}^{i}$ with $|U_{M}^{i}|=3$, we follow the same technique as used in Theorem $2.1$.\\\\
Suppose $U_{M}^{i} = \{p_{1,1}, r_{1,1}, p_{2, 1}\} \subset V(HC_{a,b,c})$. Now, by using the definition of an independent set and following the same pattern as used in Theorem $2.1$, it is simple to show that the set of vertices $U_{M}^{i}= \{p_{1,1}, r_{1,1}, p_{2, 1}\}$ forms an IMMG for $HC_{a,b,c}$ with $|U_{M}^{i}|=3$, which concludes the theorem. \\
\end{proof}

\begin{center}
  \begin{figure}[h!]
  \centering
   \includegraphics[width=2.5in]{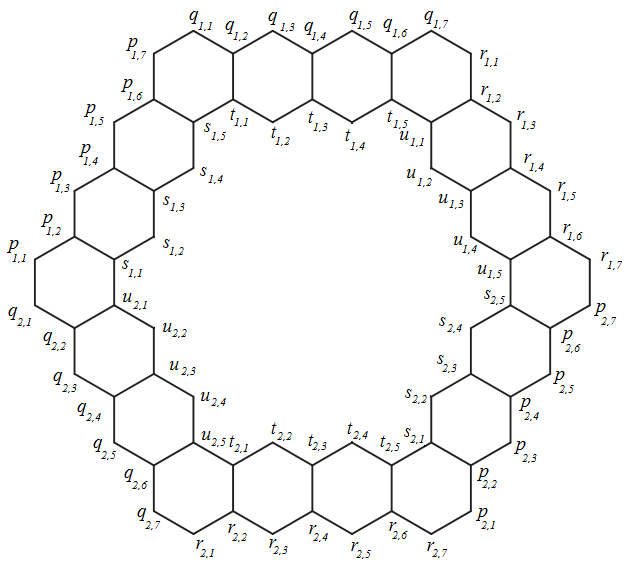}
  \caption{$HC_{4,4,4}$}\label{p2}
\end{figure}
\end{center}

\begin{exm}
If $a, b, c=4$ and $U_{M}=\{p_{1,1}, r_{1,1}, p_{2, 1}\}$ in $HC_{a,b,c}$ (as shown in Fig. 2), the mixed metric codes are as follows:\\\\
For vertices the mixed are shown in Table \ref{T2}:
\begin{table}[h]
\begin{center}
 \begin{tabular}{|c|c|c|c|c|c|c|c|}
 \hline
 Vertices   & Codes       & Vertices   & Codes    & Vertices   & Codes     & Vertices   & Codes \\
 \hline
 $p_{1,1}$  & (0,14,15)   & $q_{1,1}$  & (7,7,18) & $r_{1,1}$  & (14,0,13) & $t_{1,3}$  & (9,5,14)\\
 \hline
 $p_{1,2}$  & (1,13,14)   & $q_{1,2}$  & (8,6,17) & $r_{1,2}$  & (13,1,12) & $t_{1,4}$  & (10,4,13)\\
 \hline
 $p_{1,3}$  & (2,12,15)   & $q_{1,3}$  & (9,5,16) & $r_{1,3}$  & (14,2,11) & $t_{1,5}$  & (11,3,12)\\
 \hline
 $p_{1,4}$  & (3,11,16)   & $q_{1,4}$ & (10,4,15) & $r_{1,4}$  & (15,3,10) & $t_{2,1}$  & (8,16,7)\\
 \hline
 $p_{1,5}$  & (4,10,17)   & $q_{1,5}$ & (11,3,14) & $r_{1,5}$  & (16,4,9)  & $t_{2,2}$  & (9,15,6)\\
 \hline
 $p_{1,6}$  & (5,9,18)    & $q_{1,6}$ & (12,2,13) & $r_{1,6}$  & (17,5,8)  & $t_{2,3}$  & (10,14,5)\\
 \hline
 $p_{1,7}$  & (6,8,19)    & $q_{1,7}$ & (13,1,14) & $r_{1,7}$  & (18,6,7)  & $t_{2,4}$  & (11,13,4)\\
 \hline
 $p_{2,1}$  & (15,13,0)   & $q_{2,1}$ & (1,15,14) & $r_{2,1}$  & (8,18,7)  & $t_{2,5}$  & (12,12,3)\\
 \hline
 $p_{2,2}$  & (14,12,1)   & $q_{2,2}$ & (2,14,13) & $r_{2,2}$  & (9,17,6)  & $u_{1,1}$  & (12,2,11)\\
 \hline
 $p_{2,3}$  & (15,11,2)   & $q_{2,3}$ & (3,15,12) & $r_{2,3}$  & (10,16,5) & $u_{1,2}$  & (13,3,10)\\
 \hline
 $p_{2,4}$  & (16,10,3)   & $q_{2,4}$ & (4,16,11) & $r_{2,4}$  & (11,15,4) & $u_{1,3}$  & (14,4,9)\\
 \hline
 $p_{2,5}$  & (17,9,4)    & $q_{2,5}$ & (5,17,10) & $r_{2,5}$  & (12,14,3) & $u_{1,4}$  & (15,5,8)\\
 \hline
 $p_{2,6}$  & (18,8,5)    & $q_{2,6}$ & (6,18,9)  & $r_{2,6}$  & (13,13,2) & $u_{1,5}$  & (16,6,7)\\
 \hline
 $p_{2,7}$  & (19,7,6)    & $q_{2,7}$ & (7,19,8)  & $r_{2,7}$  & (14,14,1) & $u_{2,1}$  & (3,13,12)\\
 \hline
 $s_{1,1}$  & (2,12,13)   & $s_{1,5}$ & (6,8,17)  & $s_{2,4}$  & (16,8,5)  & $u_{2,2}$  & (4,14,11)\\
 \hline
 $s_{1,2}$  & (3,11,14)   & $s_{2,1}$ & (13,11,2) & $s_{2,5}$  & (17,7,6)  & $u_{2,3}$  & (5,15,10)\\
 \hline
 $s_{1,3}$  & (4,10,15)   & $s_{2,2}$ & (14,10,3) & $t_{1,1}$  & (7,7,16)  & $u_{2,4}$  & (6,16,9)\\
 \hline
 $s_{1,4}$  & (5,9,16)    & $s_{2,3}$ & (15,9,4)  & $t_{1,2}$  & (8,6,15)  & $u_{2,5}$  & (7,17,8)\\
 \hline
 \end{tabular}
 \caption{\label{T2}Mixed codes for the vertices of $HC_{4,4,4}$.}
 \end{center}
 \end{table}
 
\hspace{-1.3em}For edges the mixed are shown in Table \ref{T3}:
\begin{table}[h]
\begin{center}
 \begin{tabular}{|c|c|c|c|c|c|c|c|}
 \hline
 Edges   & Codes       & Edges    & Codes    & Edges    & Codes     & Edges    & Codes \\
 \hline
 $p_{1,1}p_{1,2}$  & (0,13,14)   & $q_{1,6}q_{1,7}$  & (12,1,13))& $r_{2,3}r_{2,4}$  & (10,15,4) & $t_{1,5}r_{1,1}$  & (11,2,11)\\
 \hline
 $p_{1,2}p_{1,3}$  & (1,12,14)   & $q_{1,7}r_{1,1}$  & (13,0,13) & $r_{2,4}r_{2,5}$  & (11,14,3) & $t_{2,1}t_{2,2}$  & (8,15,6)\\
 \hline
 $p_{1,3}p_{1,4}$  & (2,11,15)   & $q_{2,1}q_{2,2}$  & (1,14,13) & $r_{2,5}r_{2,6}$  & (12,13,2) & $t_{2,2}t_{2,3}$  & (9,14,5)\\
 \hline
 $p_{1,4}p_{1,5}$  & (3,10,16)   & $q_{2,2}q_{2,3}$  & (2,14,12) & $r_{2,6}r_{2,7}$  & (13,13,1) & $t_{2,3}t_{2,4}$  & (10,13,4)\\
 \hline
 $p_{1,5}p_{1,6}$  & (4,9,17)    & $q_{2,3}q_{2,4}$  & (3,15,11) & $r_{2,7}p_{2,1}$  & (14,13,0) & $t_{2,4}t_{2,5}$  & (11,12,3)\\
 \hline
 $p_{1,6}p_{1,7}$  & (5,8,18)    & $q_{2,4}q_{2,5}$  & (4,16,10) & $s_{1,1}s_{1,2}$  & (2,11,13) & $t_{2,5}s_{2,1}$  & (12,11,2)\\
 \hline
 $p_{1,7}q_{1,1}$  & (6,7,18)    & $q_{2,5}q_{2,6}$  & (5,17,9)  & $s_{1,2}s_{1,3}$  & (3,10,14) & $u_{1,1}u_{1,2}$  & (12,2,10)\\
 \hline
 $p_{2,1}p_{2,2}$  & (14,12,0)   & $q_{2,6}q_{2,7}$ & (6,18,8)  & $s_{1,3}s_{1,4}$  & (4,9,15)  & $u_{1,2}u_{1,3}$  & (13,3,9)\\
 \hline
 $p_{2,2}p_{2,3}$  & (14,11,1)   & $q_{2,7}r_{2,1}$ & (7,18,7)  & $s_{1,4}s_{1,5}$  & (5,8,16)  & $u_{1,3}u_{1,4}$  & (14,4,8) \\
 \hline
 $p_{2,3}p_{2,4}$  & (15,10,2)   & $r_{1,1}r_{1,2}$ & (13,0,12) & $s_{1,5}t_{1,1}$  & (6,7,16)  & $u_{1,4}u_{1,5}$  & (15,5,7)\\
 \hline
 $p_{2,4}p_{2,5}$  & (16,9,3)    & $r_{1,2}r_{1,3}$ & (13,1,11) & $s_{2,1}s_{2,2}$  & (13,10,2) & $u_{1,5}s_{2,5}$  & (16,6,6)\\
 \hline
 $p_{2,5}p_{2,6}$  & (17,8,4)    & $r_{1,3}r_{1,4}$ & (14,2,10) & $s_{2,2}s_{2,3}$  & (14,9,3)  & $u_{2,1}u_{2,2}$  & (3,13,11)\\
 \hline
 $p_{2,6}p_{2,7}$  & (18,7,5)    & $r_{1,4}r_{1,5}$ & (15,3,9)  & $s_{2,3}s_{2,4}$  & (15,8,4)  & $u_{2,2}u_{2,3}$  & (4,14,10)\\
 \hline
 $q_{1,1}q_{1,2}$  & (7,6,17)    & $r_{1,5}r_{1,6}$ & (16,4,8)  & $s_{2,4}s_{2,5}$  & (16,7,5) & $u_{2,3}u_{2,4}$  & (5,15,9)\\
 \hline
 $q_{1,2}q_{1,3}$  & (8,5,16)    & $r_{1,6}p_{1,7}$ & (17,5,7)  & $t_{1,1}t_{1,2}$  & (7,6,15) & $u_{2,4}u_{2,5}$  & (6,16,8)\\
 \hline
 $q_{1,3}q_{1,4}$  & (9,4,15)    & $r_{1,7}p_{2,7}$ & (18,6,6)  & $t_{1,2}t_{1,3}$  & (8,5,14) & $u_{2,5}t_{2,1}$  & (7,16,7)\\
 \hline
 $q_{1,4}q_{1,5}$  & (10,3,14)   & $r_{2,1}r_{2,2}$ & (8,17,6)  & $t_{1,3}t_{1,4}$  & (9,4,13) & $u_{2,1}s_{1,1}$  & (2,12,12)\\
 \hline
 $q_{1,5}q_{1,6}$  & (11,2,13)   & $r_{2,2}r_{2,3}$ & (9,16,5)  & $t_{1,4}t_{1,5}$  & (10,3,12)& $p_{1,2}s_{1,1}$  & (1,12,13)\\
 \hline
 $p_{1,4}s_{1,3}$  & (3,10,15)   & $p_{1,6}s_{1,5}$ & (5,8,17)  & $q_{1,2}t_{1,1}$  & (7,6,16) & $t_{1,3}q_{1,4}$  & (9,4,14)\\
 \hline
 $t_{1,5}q_{1,6}$  & (11,2,12)   & $u_{1,1}r_{1,2}$ & (12,1,11)  & $u_{1,3}r_{1,4}$ & (14,3,9) & $r_{1,6}u_{1,5}$  & (16,5,7)\\
 \hline
 $s_{2,5}p_{2,6}$  & (17,7,5)    & $s_{2,3}p_{2,4}$ & (15,9,3)  & $s_{2,1}p_{2,2}$  & (13,11,1) & $t_{2,5}r_{2,6}$  & (12,12,2)\\
 \hline
 $t_{2,3}r_{2,4}$  & (10,14,4)   & $t_{2,1}r_{2,2}$ & (8,16,6)  & $u_{2,5}q_{2,6}$  & (6,17,8) & $u_{2,3}q_{2,4}$   & (4,15,10)\\
 \hline
 $u_{2,1}q_{2,2}$  & (2,13,12)   &                  &           &                    &         &                     &\\
 \hline
 \end{tabular}
 \caption{\label{T3}Mixed codes for the edges of $HC_{4,4,4}$.}
 \end{center}
 \end{table}
 
\hspace{-1.3em}From these mixed codes in Table \ref{T2} and \ref{T3}, we find that no two elements from $V(HC_{4,4,4})\cup E(HC_{4,4,4})$ are having the same mixed codes implying $U_{M}$ to be a mixed metric generator for $HC_{4,4,4}$. Now, from these lines and by theorem 1, we have $mdim(HC_{4,4,4})=3$.
\end{exm}

\section{Minimum Multiresolving Number for $SP_{a,b,c}$}
In this section, we obtain the multiset dimension for $SP_{a,b,c}$.\\\\
\textbf{Starphene $SP_{a,b,c}$:} Starphene $SP_{a,b,c}$ is a planar graph comprises of three linear polyacenes segments consist of $a$, $b$, and $c$ number of benzene rings joined to a single central benzene ring as shown in Fig. 3. It consists of $a+b+c-2$ number of faces having six sides. Starphene $SP_{a,b,c}$ has $2(a+b+c)$ number of vertices of degree two and $2(a+b+c-3)$ number of vertices of degree three. From this, we find that $\delta(SP_{a,b,c})=2$ and $\Delta(SP_{a,b,c})=3$. The vertex set and the edge set of $SP_{a,b,c}$, are denoted by $V(SP_{a,b,c})$ and $E(SP_{a,b,c})$ respectively. Moreover, the cardinality of edges and vertices in $SP_{a,b,c}$ is given by $|E(SP_{a,b,c})|=5(a+b+c)-9$ and $|V(SP_{a,b,c})|=2(2a+2b+2c-3)$, respectively. The edge and vertex set of $SP_{a,b,c}$ are describe as follows: $V(SP_{a,b,c})=\{p_{1,g}, p_{2,g}|1\leq g\leq 2b-1\}\cup\{q_{1,g}, q_{2,g}|1\leq g\leq 2c-1\}\cup\{r_{1,g}, r_{2,g}|1\leq g\leq 2a-1\}$ \\
and\\
$E(SP_{a,b,c})=\{p_{1,g}p_{1,g+1}, p_{2,g}p_{2,g+1}|1\leq g\leq 2b-2\}\cup\{q_{1,g}q_{1,g+1}, q_{2,g}q_{2,g+1}|1\leq g\leq 2c-2\}\cup\{r_{1,g}r_{1,g+1}, r_{2,g}r_{2,g+1}|1\leq g\leq 2a-2\}\cup \{p_{1,2g-1}p_{2,2g-1}|1\leq g\leq b\}\cup\{p_{1,2g-1}p_{2,2g-1}|1\leq g\leq b\}\cup\{q_{1,2g-1}q_{2,2g-1}|1\leq g\leq c\}\cup\{r_{1,2g-1}r_{2,2g-1}|1\leq g\leq a\}\cup \{p_{1,1}r_{2,1}, r_{1,1}q_{2,1}, q_{1,1}p_{2,1}\}$.\\

We name the vertices $U_{1}=\{p_{1,g}|1\leq g \leq 2b-1\}$ in $SP_{a,b,c}$, as the $p_{1}$-vertices, the vertices $U_{2}=\{p_{2,g}|1\leq g \leq 2b-1\}$ in $SP_{a,b,c}$, as the $p_{2}$-vertices, the vertices $U_{3}=\{q_{1,g}|1\leq g \leq 2c-1\}$ in $SP_{a,b,c}$, as the $q_{1}$-vertices, the vertices $U_{4}=\{q_{2,g}|1\leq g \leq 2c-1\}$ in $SP_{a,b,c}$, as the $q_{2}$-vertices, the vertices $U_{5}=\{r_{1,g}|1\leq g \leq 2a-1\}$ in $SP_{a,b,c}$, as the $r_{1}$-vertices, and the vertices $U_{6}=\{r_{2,g}|1\leq g \leq 2a-1\}$ in $SP_{a,b,c}$, as the $r_{2}$-vertices. In vertices, $p_{1,i}$, $p_{2,i}$, $q_{1,j}$, $q_{2,j}$, $r_{1,k}$, and $r_{2,k}$, the indices $i=2b-1$, $j=2c-1$ and $k=2a-1$. Recently, results regarding vertex metric, edge metric, and mixed metric dimension of $SP_{a,b,c}$ have been reported due to  Ahmad et al. \cite{starphene}, which are as follows:
\begin{theorem}
For $a,b,c\geq 1$, we have $dim(SP_{a,b,c})=2$.
\end{theorem}
\begin{theorem}
For $a,b,c\geq 1$, we have $edim(SP_{a,b,c})=3$.
\end{theorem}
\begin{theorem}
For $a,b,c\geq 1$, we have $mdim(SP_{a,b,c})=3$.
\end{theorem}
\begin{theorem}
For $a,b,c\geq 1$, we have $pd(SP_{a,b,c})=3$.
\end{theorem}
Where $pd(SP_{a,b,c})$ is the partition dimension for $SP_{a,b,c}$ \cite{pdd}. Next, motivated from these worth findings regarding some well-known resolvability parameters for $SP_{a,b,c}$, we are interested to contribute more on this subject. Therefore, we obtain the multiset dimension for $SP_{a,b,c}$ in the next result. Before this, we have to construct sets with some properties in order to carry out this result. Here, we consider $SP_{a,b,c}$ for which we have $V(SP_{a,b,c})=\{p_{1,g}, p_{2,g}|1\leq g\leq 2b-1\}\cup\{q_{1,g}, q_{2,g}|1\leq g\leq 2c-1\}\cup\{r_{1,g}, r_{2,g}|1\leq g\leq 2a-1\}$ and a multiresolving set $U_{ms}$. We denote the sets of vertex multirepresentation for the vertices of $SP_{a,b,c}$ by $M_{1}$, $M_{2}$, $M_{3}$, $M_{4}$, $M_{5}$, and $M_{6}$, where $M_{1}=\{mr_{c}(p_{1,g}|U_{ms})=(a_{1}, a_{2}, a_{3}): a_{1}\leq a_{2}\leq a_{3}\ \&\ 1 \leq g \leq 2b-1\}$, $M_{2}=\{mr_{c}(p_{2,g}|U_{ms})=(a_{1}, a_{2}, a_{3}): a_{1}\leq a_{2}\leq a_{3}\ \&\ 1 \leq g \leq 2b-1\}$, $M_{3}=\{mr_{c}(q_{1,g}|U_{ms})=(a_{1}, a_{2}, a_{3}): a_{1}\leq a_{2}\leq a_{3}\ \&\ 1 \leq g \leq 2c-1\}$, $M_{4}=\{mr_{c}(q_{2,g}|U_{ms})=(a_{1}, a_{2}, a_{3}): a_{1}\leq a_{2}\leq a_{3}\ \&\ 1 \leq g \leq 2c-1\}$, $M_{5}=\{mr_{c}(r_{1,g}|U_{ms})=(a_{1}, a_{2}, a_{3}): a_{1}\leq a_{2}\leq a_{3}\ \&\ 1 \leq g \leq 2a-1\}$, and $M_{6}=\{mr_{c}(r_{2,g}|U_{ms})=(a_{1}, a_{2}, a_{3}): a_{1}\leq a_{2}\leq a_{3}\ \&\ 1 \leq g \leq 2a-1\}$. In the following result, we investigate the multiset dimension for $SP_{a,b,c}$.\\\\
\begin{center}
  \begin{figure}[h!]
  \centering
   \includegraphics[width=3.0in]{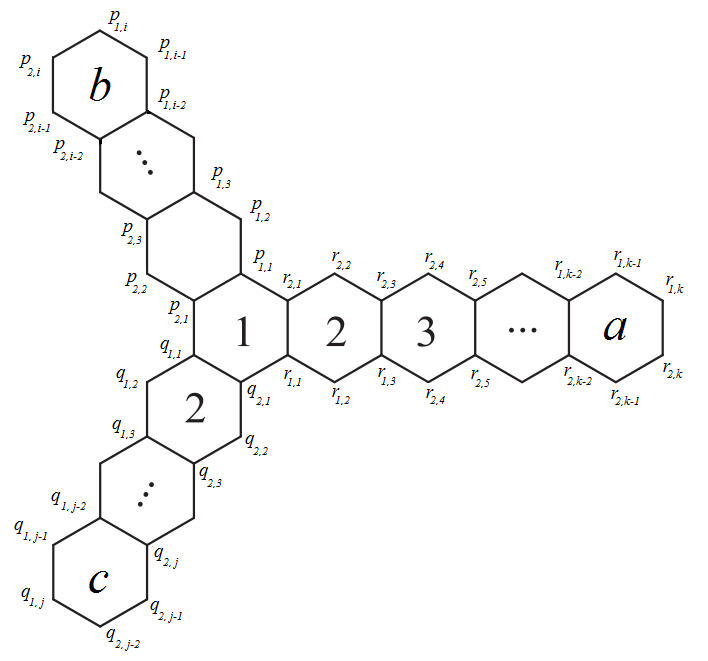}
  \caption{$SP_{a,b,c}$}\label{p2}
\end{figure}
\end{center}
\begin{theorem}
For $a,b,c\geq 3$, we have $msdim(SP_{a,b,c})=4$.
\end{theorem}
\begin{proof}
In order to show that $msdim(SP_{a,b,c}) \leq 4$, we have to construct an MRS for $SP_{a,b,c}$ having cardinality less than or equal to four. Suppose $U_{ms}=\{p_{1,i}, r_{1,k}, q_{2, 1}, q_{2, 3}\}$ be a set of distinct ordered vertices from $SP_{a,b,c}$. We claim that $U_{ms}$ is a multiresolving set for $SP_{a,b,c}$. Now, to obtain $msdim(SP_{a,b,c}) \leq 4$, we can give mutirepresentation to every element of $V(SP_{a,b,c})$ with respect to the set $U_{ms}$.\\\\
For the vertices $\{p_{1,g}|1 \leq g \leq 2b-1\}$, the mutirepresentation are as follows:
\begin{eqnarray*}
   mr_{c}(p_{1,g}|U_{ms}) = \begin{cases}
                         (2b-g-1,2b-g-2,g+1),  & 1\leq g\leq 2b-3;\\
                         (1,2,2b-1),           & g=2b-2;\\
                         (0,3,2b),             & g=2b-1.
                                             \end{cases}
\end{eqnarray*}
For the vertices $\{p_{2,g}|1 \leq g \leq 2b-1\}$, the mutirepresentation are as follows:
\begin{eqnarray*}
   mr_{c}(p_{2,g}|U_{ms}) = \begin{cases}
                         (2b-g,2b-g-3,g+2),  & 1\leq g\leq 2b-3;\\
                         (2,1,2b),           & g=2b-2;\\
                         (1,2,2b+1),         & g=2b-1.
                                             \end{cases}
\end{eqnarray*}
For the vertices $\{q_{1,g}|1 \leq g \leq 2c-1\}$, the mutirepresentation are as follows:
\begin{eqnarray*}
   mr_{c}(q_{1,g}|U_{ms}) = \begin{cases}
                         (2b+g-1,2b+g-4,g+1),   & 1\leq g\leq 2c-1.
                                             \end{cases}
\end{eqnarray*}
For the vertices $\{q_{2,g}|1 \leq g \leq 2c-1\}$, the mutirepresentation are as follows:
\begin{eqnarray*}
   mr_{c}(q_{2,g}|U_{ms}) = \begin{cases}
                         (2b+g,2b+g-3,g),       & 1\leq g\leq 2c-1.
                                             \end{cases}
\end{eqnarray*}
For the vertices $\{r_{1,g}|1\leq g \leq 2a-1\}$, the mutirepresentation are as follows:
\begin{eqnarray*}
   mr_{c}(r_{1,g}|U_{ms}) = \begin{cases}
                         (2b+g-1,2a+g-2,g-1),             & 1 \leq g \leq 2a-1.
                                             \end{cases}
\end{eqnarray*}
For the vertices $\{r_{2,g}|1 \leq g \leq 2a-1\}$, the mutirepresentation are as follows:
\begin{eqnarray*}
   mr_{c}(r_{2,g}|U_{ms}) = \begin{cases}
                         (2b+g-2,2b+g-3,g),                & 1 \leq g \leq 2a-1.
                                             \end{cases}
\end{eqnarray*}

Now, from these mutirepresentation for the vertices of $SP_{a,b,c}$, we find that $|M_{1}|=|M_{2}|=2b-1$, $|M_{3}|=|M_{4}|=2c-1$, and $|M_{5}|=|M_{2}|=2a-1$. We see that the sum of all of these cardinalities is equal to $|V(SP_{a,b,c})|$ and which is $4a+4b+4c-6$. Moreover, all of these sets of multirepresentation for $SP_{a,b,c}$ are pairwise disjoint, which implies that $msdim(SP_{a,b,c})\leq 3$. Now from this and Lemma \ref{L2}, it can be concluded that multiset dimension of $SP_{a,b,c}$ is three i.e., $msdim(SP_{a,b,c})=3$.
\end{proof}

\section{Conclusions}
In this paper, we have studied the mixed metric dimension for a hollow coronoid $HC_{a,b,c}$ and multiset dimension for starphene $SP_{a,b,c}$ structures. We have proved that $mdim(HC_{a,b,c})=3$ and $msdim(SP_{a,b,c})=3$ (a partial response to the problem raised in \cite{msd}). We also observed that the mixed resolving set for $HC_{a,b,c}$ is independent. From preliminary and our proved results, we also found that $dim(HC_{a,b,c})=edim(HC_{a,b,c})=mdim(HC_{a,b,c})=3$. This means that $HC_{a,b,c}$ is the families of graphs having the same constant metric, edge metric and mixed metric dimension. In future, we will try to obtain some other variants of metric dimension for the graphs $HC_{a,b,c}$ and $SP_{a,b,c}$.

\end{document}